\documentclass[12pt,reqno]{amsart}
\usepackage{amsfonts}
\usepackage{amssymb,color}
\usepackage{amssymb}

\usepackage[colorlinks=true]{hyperref}
\hypersetup{urlcolor=blue, citecolor=blue}

\usepackage[margin=3cm, a4paper]{geometry}

\newtheorem{theo}{Theorem}[section]
\newtheorem{lemm}[theo]{Lemma}
\newtheorem{defi}[theo]{Definition}
\newtheorem{coro}[theo]{Corollary}

\newtheorem{rema}[theo]{Remark}
\numberwithin{equation}{section}

\newcommand{\bal}{\begin{align}}
\newcommand{\bbal}{\begin{align*}}
\newcommand{\beq}{\begin{equation}}
\newcommand{\eeq}{\end{equation}}
\newcommand{\bca}{\begin{cases}}
\newcommand{\eca}{\end{cases}}

\newcommand{\pa}{\partial}

\newcommand{\na}{\nabla}
\newcommand{\De}{\Delta}

\newcommand{\cd}{\cdot}

\newcommand{\dd}{\mathrm{d}}

\newcommand{\R}{\mathbb{R}}

\newcommand{\D}{\mathrm{div}}
\newcommand{\Z}{\mathbb{Z}}

\begin{document}


\title[Euler-Poincar\'{e} system]{Non-uniform continuous dependence on initial data of solutions to the Euler-Poincar\'{e} system}

\author[J. Li]{Jinlu Li}
\address{School of Mathematics and Computer Sciences, Gannan Normal University, Ganzhou 341000, China}
\email{lijl29@mail2.sysu.edu.cn}

\author[L. Dai]{Li Dai}
\address{School of Mathematics and Computer Sciences, Gannan Normal University, Ganzhou 341000, China}
\email{daili221726@163.com}

\author[W. Zhu]{Weipeng Zhu}
\address{School of Mathematics and Information Science, Guangzhou University, Guangzhou, 510006, China}
\email{mathzwp2010@163.com}

\begin{abstract}
In this paper, we investigate the continuous dependence on initial data of solutions to the Euler-Poincar\'{e} system. By constructing a sequence approximate solutions and calculating the error terms, we show that the data-to-solution map is not uniformly continuous in Sobolev space $H^s(\R^d)$ for $s>1+\frac d2$.
\end{abstract}

\maketitle

\section{Introduction and main result}

In the paper, we consider the following Cauchy problem of the Euler-Poincar\'{e} system:
\begin{equation}\label{E-P}
\begin{cases}
\partial_tm+u\cdot \nabla m+(\nabla u)^Tm+(\mathrm{div} u)m=0, \qquad (t,x)\in \R^+\times \R^d,\\
m=(1-\De)u,\qquad (t,x)\in \R^+\times \R^d,\\
u(0,x)=u_0,\qquad x\in \R^d,
\end{cases}
\end{equation}
where $u=(u_1,u_2,\cdots,u_d)$ denotes the velocity of the fluid, $m=(m_1,m_2,\cdots,m_d)$ represents the momentum. The notation $(\nabla u)^T$ represents the transpose of the matrix $\nabla u$.

The system \eqref{E-P} is the classical Camassa-Holm (CH) equation for $d=1$, while it is also called the Euler-Poincar\'{e} equations in the higher dimensional case $d\geq1$. The Camassa-Holm equation can be regarded as a shallow water wave equation \cite{Camassa,Camassa.Hyman,Constantin.Lannes}.  It is completely integrable \cite{Camassa,Constantin-P}, has a bi-Hamiltonian structure \cite{Constantin-E,Fokas}, and admits exact peaked solitons of the form $ce^{-|x-ct|}$, $c>0$, which are orbitally stable \cite{Constantin.Strauss}. We have to say that the peaked solitons present the characteristic for the travelling water waves of greatest height and largest amplitude and arise as solutions to the free-boundary problem for incompressible Euler equations over a flat bed, cf., \cite{Constantin2,Constantin.Escher4,Constantin.Escher5,Toland} and references therein. The local well-posedness and ill-posedness for the Cauchy problem of the CH equation in Sobolev spaces and Besov spaces was discussed in \cite{Constantin.Escher,Constantin.Escher2,d1,d3,Guo-Yin,Li-Yin,Ro}. The existence and finite time blow-up of strong solutions to the CH equation was shown in \cite{Constantin,Constantin.Escher,Constantin.Escher2,Constantin.Escher3}. For the existence and uniqueness of global weak solutions to the CH equation, we refer the reader to see \cite{Constantin.Molinet, Xin.Z.P}. The global conservative and dissipative solutions of CH equation were studied in \cite{Bressan.Constantin,Bressan.Constantin2}. The non-uniform dependence on initial data for the CH equation was discussed in \cite{H-K,H-K-M}.

For the Euler-Poincar\'{e} system, it was first theoretically studied by Chae and Liu in the pioneering work \cite{Chae.Liu}. The authors obtained the local well-posedness in Hilbert spaces $m_0\in H^{s+\frac d2},\ s\geq2$ and also gave a  blow-up criterion, zero $\alpha$ limit and the Liouville type theorem. In \cite{L.Y.Z}, Li, Yu and Zhai showed that the solution to \eqref{E-P} with a large class of smooth initial data blows up in finite time or exists globally in time, which reveals the nonlinear depletion mechanism hidden in the Euler-Poincar\'{e} system. By applying Littlewood-Paley theory, the local existence and uniqueness in Besov spaces $B^s_{p,r},\ s>\max\{\frac32,1+\frac dp\}$ and $s=1+\frac dp,\ 1\leq p\leq 2d,\ r=1$, were established by Yan and Yin \cite{Y.Y}.
Lately, Li and Yin \cite{Li-Yin1} proved that the corresponding solution is continuous dependence for the initial data in Besov spaces. Zhao, Yang and Li \cite{Z.Y.L} showed that the solution map of the periodic Euler-Poincar\'{e} system is not uniformly continuous in Besov space $B^s_{2,r},s>1+\frac d2$.

In this paper, inspired by \cite{H-H,H-M},  we will show that the solution map of \eqref{E-P} is not uniformly continuous dependence in Sobolev space $H^s(\R^d)$. Compared with the Camassa-Holm type equation in one dimension, we need choose the suitable approximate solutions and calculate the more error terms.

According to \cite{Y.Y}, we can transform \eqref{E-P} into the following form:
\begin{align}\label{E-P1}
\partial_tu+u\cdot \nabla u=P(u,u):=-(1-\De)^{-1}\D Q(u,u)-(1-\De)^{-1}R(u,u),
\end{align}
where
\bbal
&Q(u,v)=\nabla u\nabla v+\nabla u(\nabla v)^T-(\nabla u)^T\nabla v-(\mathrm{div} u)\nabla v+\frac12\mathbf{I}(\nabla u:\nabla v),
\\&R(u,v)=(\mathrm{div} u)v+ u\cd \na v.
\end{align*}

Then, we have the following result.

\begin{theo}\label{th2}
Let $d\geq 2$ and $s>1+\frac d2$. The data-to-solution map for the Euler-Poincar\'{e} system \eqref{E-P} is not uniformly continuous from any bounded subset in $H^s$ into $\mathcal{C}([0,T];H^s(\R^d))$. That is, there exists two sequences of solutions $u^n$ and $v^n$ such that
\bbal
&||u^n_0||_{H^s}+||v^n_0||_{H^s(\R^d)}\lesssim 1, \quad \lim_{n\rightarrow \infty}||u^n_0-v^n_0||_{H^s(\R^d)}= 0,
\\&\liminf_{n\rightarrow \infty}||u^n(t)-v^n(t)||_{H^s(\R^d)}\gtrsim |\sin t|,  \quad t\in[0,T].
\end{align*}
\end{theo}

Our paper is organized as follows. In Section 2, we give some preliminaries which will be used in the sequel. In Section 3, we give the proof of our main theorem.\\

\noindent\textbf{Notations.} Given a Banach space $X$, we denote its norm by $\|\cdot\|_{X}$. The symbol $A\lesssim B$ means that there is a uniform positive constant $c$ independent of $A$ and $B$ such that $A\leq cB$. Here
\bbal
&(\na u)_{i,j}=\pa_{x_i}u^j, \quad (u\cd \na v)_i=\sum^d_{k=1}u_k\pa_{x_k}u_j, \quad (\na u\na v)_{ij}=\sum^d_{k=1}\pa_{x_i}u_k\pa_{x_k}v_j,
\\&\na u:\na v=\sum^d_{i,j=1}\pa_{x_i}u_j\pa_{x_i}v_j.
\end{align*}

\section{Littlewood-Paley analysis}

In this section, we will recall some facts about the Littlewood-Paley decomposition, the nonhomogeneous Besov spaces and their some useful properties. For more details, the readers can refer to \cite{B.C.D}.

There exists a couple of smooth functions $(\chi,\varphi)$ valued in $[0,1]$, such that $\chi$ is supported in the ball $\mathcal{B}\triangleq \{\xi\in\mathbb{R}^d:|\xi|\leq \frac 4 3\}$, and $\varphi$ is supported in the ring $\mathcal{C}\triangleq \{\xi\in\mathbb{R}^d:\frac 3 4\leq|\xi|\leq \frac 8 3\}$. Moreover,
$$\forall\,\ \xi\in\mathbb{R}^d,\,\ \chi(\xi)+{\sum\limits_{j\geq0}\varphi(2^{-j}\xi)}=1,$$
$$\forall\,\ 0\neq\xi\in\mathbb{R}^d,\,\ {\sum\limits_{j\in \Z}\varphi(2^{-j}\xi)}=1,$$
$$|j-j'|\geq 2\Rightarrow\textrm{Supp}\,\ \varphi(2^{-j}\cdot)\cap \textrm{Supp}\,\ \varphi(2^{-j'}\cdot)=\emptyset,$$
$$j\geq 1\Rightarrow\textrm{Supp}\,\ \chi(\cdot)\cap \textrm{Supp}\,\ \varphi(2^{-j}\cdot)=\emptyset,$$
Then, we can define the nonhomogeneous dyadic blocks $\Delta_j$ and nonhomogeneous low frequency cut-off operator $S_j$ as follows:
$$\Delta_j{u}= 0,\,\ if\,\ j\leq -2,\quad
\Delta_{-1}{u}= \chi(D)u=\mathcal{F}^{-1}(\chi \mathcal{F}u),$$
$$\Delta_j{u}= \varphi(2^{-j}D)u=\mathcal{F}^{-1}(\varphi(2^{-j}\cdot)\mathcal{F}u),\,\ if \,\ j\geq 0,$$
$$S_j{u}= {\sum\limits_{j'=-\infty}^{j-1}}\Delta_{j'}{u}.$$

\begin{defi}(\cite{B.C.D})\label{de2.3}
Let $s\in\mathbb{R}$ and $1\leq p,r\leq\infty$. The nonhomogeneous Besov space $B^s_{p,r}$ consists of all tempered distribution $u$ such that
\begin{align*}
||u||_{B^s_{p,r}(\R^d)}\triangleq \Big|\Big|(2^{js}||\Delta_j{u}||_{L^p(\R^d)})_{j\in \Z}\Big|\Big|_{\ell^r(\Z)}<\infty.
\end{align*}
\end{defi}
\begin{rema}(\cite{B.C.D})
When $p=r=2$, we have $B^s_{2,2}(\R^d)=H^s(\R^d)$. Here, $H^s(\R^d)$ is the standard Sobolev space with the norm
\bbal
||u||^2_{H^s(\R^d)}:=\int_{\R^d}(1+|\xi^2|)^s|\hat{u}(\xi)|^2\dd \xi.
\end{align*}
For any $u\in B^s_{2,2}(\R^d)$, there holds
\bbal
c_0||u||_{H^s(\R^d)}\leq ||u||_{B^s_{2,2}(\R^d)}\leq C_0||u||_{H^s(\R^d)}.
\end{align*}
If $s>\frac d2$, we also have $||u||_{L^\infty(\R^d)}\lesssim ||u||_{B^s_{2,2}(\R^d)}$.
\end{rema}

Then, we have the following product laws.
\begin{lemm}(\cite{B.C.D})\label{le1}
Let $d\geq 2$ and $s>\frac d2$. Then there exists a constant $C=C(d,s)$ such that
\[||uv||_{B^{s-1}_{2,2}(\R^d)}\leq C\big(||u||_{L^\infty(\R^d)}||v||_{B^{s-1}_{2,2}(\R^d)}+||v||_{L^\infty(\R^d)}||u||_{B^{s-1}_{2,2}(\R^d)}\big),\]
\[||uv||_{B^s_{2,2}(\R^d)}\leq C||u||_{B^s_{2,2}(\R^d)}||v||_{B^s_{2,2}(\R^d)}, \quad ||uv||_{B^{s-1}_{2,2}(\R^d)}\leq C||u||_{B^{s-1}_{2,2}(\R^d)}||v||_{B^s_{2,2}(\R^d)}.\]
\end{lemm}

\begin{lemm}(\cite{B.C.D})\label{le2}
Let $d\geq 2$ and $s>\frac d2-1$. Assume that $f_0\in B^{s}_{2,2}$, $F\in L^1_T(B^{s}_{2,2}(\R^d))$ and $\nabla v\in L^1_T(B^{s-1}_{2,2}(\R^d)\cap B^{\frac d2}_{2,2}(\R^d)\cap L^\infty(\R^d))$. If $f\in \mathcal{C}([0,T];B^{s}_{2,2}(\R^d))$ solves the following linear transport equation:
\begin{equation*}
\begin{cases}
\pa_tf+v\cdot \nabla f=F,\\
f(0,x)=f_0,
\end{cases}
\end{equation*}
then there exists a positive constant $C=C(s)$ such that
$$||f||_{B^{s}_{2,2}(\R^d)}\leq ||f_0||_{B^{s}_{2,2}(\R^d)}+C\int^t_0V'(\tau)||f(\tau)||_{B^{s}_{2,2}(\R^d)}\mathrm{d}\tau+\int^t_0||F(\tau)||_{B^{s}_{2,2}(\R^d)}\mathrm{d}\tau,$$
or
$$||f||_{B^{s}_{2,2}(\R^d)}\leq e^{CV(t)}\Big(||f_0||_{B^{s}_{2,2}(\R^d)}+\int^t_0e^{-CV(\tau)}||F(\tau)||_{B^{s}_{2,2}(\R^d)}\mathrm{d}\tau\Big),$$
where
\begin{equation*}
V(t)=
\begin{cases}
\int^t_0||\nabla v||_{B^{s-1}_{2,2}(\R^d)} \ \mathrm{d}\tau, \quad s>\frac d2,\\
\int^t_0||\nabla v||_{B^{\frac d2}_{2,2}(\R^d)\cap L^\infty(\R^d)} \ \mathrm{d}\tau, \quad s\leq\frac d2.
\end{cases}
\end{equation*}
\end{lemm}

\begin{lemm}(\cite{H-H})\label{le3}
Let $\phi\in S(\mathbb{R})$, $\delta>0$ and $s\geq 0$ . Then we have for all $\alpha\in \mathbb{R}$,
\bbal
&\lim\limits_{n\rightarrow \infty}n^{-\frac12\delta-s}||\phi(\frac{x}{n^{\delta}})\cos(nx-\alpha)||_{H^s(\R)}=\frac{1}{\sqrt{2}}||\phi||_{L^2(\R)},
\\&\lim\limits_{n\rightarrow \infty}n^{-\frac12\delta-s}||\phi(\frac{x}{n^{\delta}})\sin(nx-\alpha)||_{H^s(\R)}=\frac{1}{\sqrt{2}}||\phi||_{L^2(\R)}.
\end{align*}
\end{lemm}

\section{Non-uniform continuous dependence}

In this section, we will give the proof of our main theorem. Firstly, motivated by \cite{H-H,H-M}, we can construct a sequence approximate solutions where the last $d-1$ component is 0. Lately, we consider the difference about approximate solution and actual solution and also show that this distance is decaying. Finally, by the precious steps, we can conclude that the the solution map is not uniformly continuous. In order to state our main result, we first recall the following local-in-time existence of strong solutions to \eqref{E-P} in \cite{Y.Y}:
\begin{lemm}(\cite{Y.Y})\label{le10}
For $s>1+\frac d2,\ d\geq 2$, and initial data $u_0\in B^s_{2,2}(\R^d)$, there exists a time $T=T(s,d,||u_0||_{B^s_{2,2}(\R^d)})>0$ such that the system \eqref{E-P} have a unique solution $u\in \mathcal{C}([0,T];B^s_{2,2}(\R^d))$. Moreover, for all $t\in[0,T]$, there holds
\[||u(t)||_{B^s_{2,2}(\R^d)}\leq \frac{||u_0||_{B^s_{2,2}(\R^d)}}{1-C_sT||u_0||_{B^s_{2,2}(\R^d)}}.\]
\end{lemm}

\begin{coro}\label{co1}
Let $s>1+\frac d2,\ d\geq 2$. Let $u\in C([0,T];H^s)$ be the solution of the system \eqref{E-P}. Then, we have for all $t\in[0,T]$,
\bbal
||u(t)||_{B^{s-1}_{2,2}(\R^d)}\leq ||u_0||_{B^{s-1}_{2,2}(\R^d)}e^{\int^t_0||u(\tau)||_{B^s_{2,2}(\R^d)}\dd \tau},
\end{align*}
and
\bbal
||u(t)||_{B^{s+1}_{2,2}(\R^d)}\leq ||u_0||_{B^{s+1}_{2,2}(\R^d)}e^{\int^t_0||u(\tau)||_{B^s_{2,2}(\R^d)}\dd \tau}.
\end{align*}
\end{coro}
\begin{proof}
The results can easily deduce from Lemma \ref{le2} and Gronwall's inequality. Here, we omit it.
\end{proof}

Now, we give the details of the proof to our theorem. \\
\textbf{Proof of the main theorem.} Set $0<\delta<\frac12$  and let $\phi$ be a $C_0(\mathbb{R})$ function such that
\begin{equation*}
\phi(x)=
\begin{cases}
1, \quad |x|\leq 1,\\
0, \quad |x|\geq 2.
\end{cases}
\end{equation*}
Moreover, let $\Phi$ be a $C_0(\mathbb{R})$ function such that $\Phi(x)\phi(x)=\phi(x)$. Firstly, we choose the velocity having the following form:
\bbal
u^{\omega,n}=u^{l,\omega,n}(t,x)+u^{h,\omega,n}(t,x), \qquad \omega\in\{\pm1\},\ x\in\R^d,\ t\in \R,
\end{align*}
where $u^{h,\omega,n}$ is the high frequency term
\bbal
u^{h,\omega,n}(t,x)=
\Big(\phi^{h,\omega,n}(t,x),0,\cdots,0\Big),
\end{align*}
with
\bbal
\phi^{h,\omega,n}(t,x)=
n^{-\frac12\delta-s}\phi(\frac{x_1}{n^\delta})\sin(n x_1-\omega t)\phi(x_2)\cdots \phi(x_d),
\qquad n \in \Z.
\end{align*}
To choose the suitable low frequency term $u^{l,\omega,n}$, we let $u^{l,\omega,n}$ satisfy the following initial value problem:
\begin{align}\label{equa-1}\begin{cases}
\partial_tu^{l,\omega,n}+u^{l,\omega,n}\cdot \nabla u^{l,\omega,n}=P(u^{l,\omega,n},u^{l,\omega,n}),\\
u^{l,\omega,n}(0,x)=
\big(\omega n^{-1}\Phi(\frac{x_1}{n^{\delta}}){\Phi(x_2)\cdots \Phi(x_d)},0,\cdots,0\big).
\end{cases}\end{align}
By the well-posedness result (see Lemma \ref{le10}), $u^{l,\omega,n}$ belong to $C([0,T];B^s_{2,2})$ and have lifespan $T\simeq 1$. In order to simplify the notation, we set $\widetilde{\phi}(x_2,\cdots,x_d)=\phi(x_2)\cdots \phi(x_d)$ and  ${\widetilde{\Phi}(x_2,\cdots,x_d)}={\Phi(x_2)\cdots \Phi(x_d)}$. Thus, we can find that the constructional solution $u^{\omega,n}$ satisfies the following equations:
\begin{footnotesize}
\bbal
&\quad \partial_tu^{\omega,n}+u^{\omega,n}\cdot \nabla u^{\omega,n}=P(u^{\omega,n},u^{\omega,n})+\underbrace{\partial_tu^{h,\omega,n}+u^{l,\omega,n}\cdot \nabla u^{h,\omega,n}}_{E^{\omega,n}}
\\&+\underbrace{u^{h,\omega,n}\cdot \nabla u^{h,\omega,n}+u^{h,\omega,n}\cdot \nabla u^{l,\omega,n}-P(u^{l,\omega,n},u^{h,\omega,n})-P(u^{h,\omega,n},u^{l,\omega,n})-P(u^{h,\omega,n},u^{h,\omega,n})}_{F^{\omega,n}},
\end{align*}
\end{footnotesize}
with initial data
\begin{small}
\bbal
u^{\omega,n}_0:&=u^{\omega,n}(0,x) \\
&=(n^{-\frac12\delta-s}\phi(\frac{x_1}{n^\delta})\sin(n x_1-\omega t)\widetilde{\phi}(x_2,\cdots,x_d)
-\omega n^{-1}\Phi(\frac{x_1}{n^{\delta}}){\widetilde{\Phi}(x_2,\cdots,x_d)},0,\cdots,0\big).
\end{align*}
\end{small}
Let us consider the actual solution $u_{\omega,n}$ which is the solution of \eqref{E-P1} with the same initial data $u^{\omega,n}_0$. Then, $u_{\omega,n}$ satisfies
\bbal
\partial_tu_{\omega,n}+u_{\omega,n}\cdot \nabla u_{\omega,n}&=P(u_{\omega,n},u_{\omega,n}),
\end{align*}
with initial data
\begin{small}
\bbal
u_{\omega,n}(0,x)=(n^{-\frac12\delta-s}\phi(\frac{x_1}{n^\delta})\sin(n x_1-\omega t)\widetilde{\phi}(x_2,\cdots,x_d)
-\omega n^{-1}\Phi(\frac{x_1}{n^{\delta}}){\widetilde{\Phi}(x_2,\cdots,x_d)},0,\cdots,0\big).
\end{align*}
\end{small}
By the well-posedness result (see Lemma \ref{le10}), the solution $u_{\omega,n}$ belong to $C([0,T];B^s_{2,2})$ and have common lifespan $T\simeq 1$. For the estimates of $u_{\omega,n}$, we get from Corollary \ref{co1} that
\bal\label{eq4}\begin{split}
&||u_{\omega,n}||_{L^\infty_T(B^{s-1}_{2,2}(\R^d))}\lesssim n^{-1+\frac12\delta}, \  ||u_{\omega,n}||_{L^\infty_T(B^{s}_{2,2}(\R^d))}\lesssim  1, \ ||u_{\omega,n}||_{L^\infty_T(B^{s+1}_{2,2}(\R^d))}\lesssim n.
\end{split}\end{align}
Next, considering the difference $v=u_{\omega,n}-u^{\omega,n}$, we observe that $v$ satisfy
\begin{align*}
\partial_tv+u_{\omega,n}\cd\na v +v \cd\na u^{\omega,n}&=P(u_{\omega,n},v)+P(v,u^{\omega,n})-E^{\omega,n}-F^{\omega,n},
\end{align*}
with initial data $v_0=0$.\\

Now, we shall estimate the difference $v$ in the Sobolev $B^{s-1}_{2,2}$ norm. We hope that the decay of $||v||_{B^{s-1}_{2,2}}$ is less than $n^{-1}$.  According to Lemmas \ref{le1}-\ref{le2}, we have for all $t\in[0,T]$,
\bal\label{eq-sum}\begin{split}
||v(t)||_{B^{s-1}_{2,2}(\R^d)}&\leq C\int^t_0||v||_{B^{s-1}_{2,2}}\big(||u^{\omega,n}||_{B^{s}_{2,2}(\R^d)}+||u_{\omega,n}||_{B^{s}_{2,2}(\R^d)}\big)\dd \tau
\\&\quad +C\int^t_0\big(||E^{\omega,n}||_{B^{s-1}_{2,2}(\R^d)}+||F^{\omega,n}||_{B^{s-1}_{2,2}(\R^d)}\big)\dd \tau.
\end{split}\end{align}
Hence, we shall estimate the terms $E^{\omega,n}$ and $F^{\omega,n}$ in Sobolev $B^{s-1}_{2,2}$ norm. To obtain the desired result, we need to estimate the terms $u^{l,\omega,n}$ and $u^{h,\omega,n}$. By Lemma \ref{le3} and Corollary \ref{co1} , we have for any $r\geq 0$ and $t\in[0,T]$,
\bal\label{eq4-0}\begin{split}
&||u^{l,\omega,n}(t)||_{B^{s+1}_{2,2}(\R^d)}\lesssim ||u^{l,\omega,n}_0||_{B^{s+1}_{2,2}(\R^d)}\lesssim n^{-1+\frac12\delta}, \quad ||u^{h,\omega,n}(t)||_{B^r_{2,2}(\R^d)}\lesssim n^{r-s},
\\& ||u^{h,\omega,n}(t)||_{L^\infty(\R^d)}\lesssim n^{-2}, \qquad ||\na u^{h,\omega,n}(t)||_{L^\infty(\R^d)}\lesssim n^{-1}.
\end{split}\end{align}
Collecting the following estimates
\bbal
||u^{h,\omega,n}\cdot \nabla u^{l,\omega,n}||_{B^{s-1}_{2,2}(\R^d)}
\lesssim ||u^{h,\omega,n}||_{B^{s-1}_{2,2}(\R^d)}||u^{l,\omega,n}||_{B^{s}_{2,2}(\R^d)}\lesssim n^{-2+\frac12\delta},
\end{align*}
\bbal
||u^{h,\omega,n}\cdot \nabla u^{h,\omega,n}||_{B^{s-1}_{2,2}(\R^d)}
\lesssim &||u^{h,\omega,n}||_{B^{s-1}_{2,2}(\R^d)}||\nabla u^{h,\omega,n}||_{L^\infty(\R^d)}
\\& \ +||u^{h,\omega,n}||_{B^{s}_{2,2}(\R^d)}||u^{h,\omega,n}||_{L^\infty(\R^d)}\lesssim n^{-2},
\end{align*}
\bbal
||Q(u^{h,\omega,n},u^{h,\omega,n})||_{B^{s-2}_{2,2}(\R^d)}
\lesssim ||u^{h,\omega,n}||_{B^{s-1}_{2,2}(\R^d)}||\nabla u^{h,\omega,n}||_{L^\infty(\R^d)}\lesssim n^{-2},
\end{align*}
\bbal
||Q(u^{l,\omega,n},u^{h,\omega,n})||_{B^{s-2}_{2,2}(\R^d)}
\lesssim ||u^{l,\omega,n}||_{B^{s}_{2,2}(\R^d)}||u^{h,\omega,n}||_{B^{s-1}_{2,2}(\R^d)}\lesssim n^{-2+\frac12\delta},
\end{align*}
\bbal
||Q(u^{h,\omega,n},u^{l,\omega,n})||_{B^{s-2}_{2,2}(\R^d)}
\lesssim ||u^{l,\omega,n}||_{B^{s}_{2,2}(\R^d)}||u^{h,\omega,n}||_{B^{s-1}_{2,2}(\R^d)}\lesssim n^{-2+\frac12\delta},
\end{align*}
\bbal
||R(u^{h,\omega,n},u^{h,\omega,n})||_{B^{s-2}_{2,2}(\R^d)}
\lesssim ||u^{h,\omega,n}||_{B^{s-1}_{2,2}(\R^d)}||u^{h,\omega,n}||_{B^{s-1}_{2,2}(\R^d)}\lesssim n^{-2},
\end{align*}
\bbal
||R(u^{l,\omega,n},u^{h,\omega,n})||_{B^{s-2}_{2,2}(\R^d)}
\lesssim ||u^{l,\omega,n}||_{B^{s-1}_{2,2}(\R^d)}||u^{h,\omega,n}||_{B^{s-1}_{2,2}(\R^d)}\lesssim n^{-2+\frac12\delta},
\end{align*}
\bbal
||R(u^{h,\omega,n},u^{l,\omega,n})||_{B^{s-2}_{2,2}(\R^d)}
\lesssim ||u^{l,\omega,n}||_{B^{s-1}_{2,2}(\R^d)}||u^{h,\omega,n}||_{B^{s-1}_{2,2}(\R^d)}\lesssim n^{-2+\frac12\delta},
\end{align*}
we can infer that
\bal\label{eq4-1}\begin{split}
||F^{\omega,n}||_{B^{s-1}_{2,2}(\R^d)}\lesssim n^{-2+\frac12\delta}.
\end{split}\end{align}
For the term $E^{\omega,n}$, we can estimate it by component. Direct calculation shows that $\big(E^{\omega,n}\big)_i=0$ for $i=2,\cdots d$. According to the definition of $u^{l,\omega,n}(0,x)$, we can write the first component of $\pa_tu^{h,\omega,n}$ and $u^{l,\omega,n}\cd \na u^{h,\omega,n}$ in the form
\bbal
\Big(\pa_tu^{h,\omega,n}\Big)_1(t,x)&=-\omega n^{-\frac12\delta-s}\phi(\frac{x_1}{n^\delta})\cos(n x_1-\omega t)\widetilde{\phi}(x_2,\cdots,x_d)
\\&=-n^{-s+1-\frac12\delta}u^{l,\omega,n}_1(0,x)\phi(\frac{x_1}{n^\delta})\cos(nx_1-\omega t)\widetilde{\phi}(x_2,\cdots,x_d),
\end{align*}
and
\bbal
&\quad\Big(u^{l,\omega,n}\cd \na u^{h,\omega,n}\Big)_1(t,x)
\\&=n^{-s+1-\frac12\delta}u^{l,\omega,n}_1(t,x)\phi(\frac{x_1}{n^\delta})\cos(nx_1-\omega t)\widetilde{\phi}(x_2,\cdots,x_d)
\\&\quad +n^{-s-\frac32\delta}u^{l,\omega,n}_1(t,x)\phi'(\frac{x_1}{n^\delta})\sin(nx_1-\omega t)\widetilde{\phi}(x_2,\cdots,x_d)
\\&\quad +n^{-\frac12\delta-s}\phi(\frac{x_1}{n^\delta})\sin(nx_1-\omega t)\sum^d_{i=2}u^{l,\omega,n}_i(t,x)\pa_{x_i}\widetilde{\phi}(x_2,\cdots,x_d).
\end{align*}
Therefore, the term $\big(E^{\omega,n}\big)_1$ can be written as
\bbal
&\quad\Big(\pa_tu^{h,\omega,n}+u^{l,\omega,n}\cd \na u^{h,\omega,n}\Big)_1(t,x)
\\&=n^{-s+1-\frac12\delta}[u^{l,\omega,n}_1(t,x)-u^{l,\omega,n}_1(0,x)]\phi(\frac{x_1}{n^\delta})\cos(nx_1-\omega t)\widetilde{\phi}(x_2,\cdots,x_d)
\\&\quad -n^{-s-\frac32\delta}u^{l,\omega,n}_1(t,x)\phi'(\frac{x_1}{n^\delta})\sin(nx_1-\omega t)\widetilde{\phi}(x_2,\cdots,x_d)
\\&\quad +n^{-\frac12\delta-s}\phi(\frac{x_1}{n^\delta})\sin(nx_1-\omega t)\sum^d_{i=2}u^{l,\omega,n}_i(t,x)\pa_{x_i}\widetilde{\phi}(x_2,\cdots,x_d)
\\&:=E^{\omega,n}_{1,1}(t)+E^{\omega,n}_{1,2}(t)+E^{\omega,n}_{1,3}(t).
\end{align*}
By Lemmas \ref{le1}-\ref{le3}, we can estimate the last two terms as follows:
\bal\label{eq5-1}\begin{split}
||E^{\omega,n}_{1,2}(t)||_{B^{s-1}_{2,2}(\R^d)}
&\lesssim n^{-s-\frac32\delta}||u^{l,\omega,n}_1(t,x)||_{B^{s}_{2,2}(\R^d)}
\\& \quad \times ||\phi'(\frac{x_1}{n^\delta})\sin(nx_1-\omega t)||_{B^{s-1}_{2,2}(\R)}||\widetilde{\phi}||_{B^{s-1}_{2,2}(\R^{d-1})}
\\&\lesssim  n^{-s-\frac32\delta}\cdot n^{-1+\frac12\delta}\cdot n^{s-1+\frac12\delta}\lesssim n^{-2},
\end{split}\end{align}
and
\bal\label{eq5-2}\begin{split}
||E^{\omega,n}_{1,3}(t)||_{B^{s-1}_{2,2}(\R^d)}
&\lesssim n^{-s-\frac12\delta}||u^{l,\omega,n}(t,x)||_{B^{s}_{2,2}(\R^d)}
\\&\quad \times ||\phi(\frac{x_1}{n^\delta})\sin(nx_1-\omega t)||_{B^{s-1}_{2,2}(\R)}||\widetilde{\phi}||_{B^{s}_{2,2}(\R^{d-1})}
\\&\lesssim n^{-s-\frac12\delta}\cdot n^{-1+\frac12\delta}\cdot n^{s-1+\frac12\delta} \lesssim n^{-2+\frac12\delta}.
\end{split}\end{align}
For the term $E^{\omega,n}_{1,1}(t)$, by Lemmas \ref{le1}-\ref{le3}, we can compute it as
\bal\label{eq5-3}\begin{split}
||E^{\omega,n}_{1,1}(t)||_{B^{s-1}_{2,2}(\R^d)}
&\lesssim n^{-s+1-\frac12\delta}||u^{l,\omega,n}(t)-u^{l,\omega,n}(0)||_{B^{s}_{2,2}(\R^d)}
\\&\quad \times||\phi(\frac{x_1}{n^\delta})\sin(nx_1-\omega t)||_{B^{s-1}_{2,2}(\R)}||\widetilde{\phi}||_{B^{s}_{2,2}(\R^{d-1})}
\\&\lesssim  n^{-s+1-\frac12\delta}\cd  n^{s-1+\frac12\delta} ||u^{l,\omega,n}(t)-u^{l,\omega,n}(0)||_{B^{s}_{2,2}(\R^d)}
\\&\lesssim  ||u^{l,\omega,n}(t)-u^{l,\omega,n}(0)||_{B^{s}_{2,2}(\R^d)}.
\end{split}\end{align}
Since $u^{l,\omega,n}$ is the solution of \eqref{equa-1}, then we can estimate the integral term above by
\bal\label{eq5-4}\begin{split}
&\quad ||u^{l,\omega,n}(t)-u^{l,\omega,n}(0)||_{B^{s}_{2,2}(\R^d)}
\\&\leq \int^t_0||\pa_\tau u^{l,\omega,n}(\tau)||_{B^{s}_{2,2}(\R^d)}\dd \tau
\\&\leq\int^t_0\big(||u^{l,\omega,n}\cd \na u^{l,\omega,n}||_{B^{s}_{2,2}(\R^d)}+||P(u^{l,\omega,n},u^{l,\omega,n})||_{B^{s}_{2,2}(\R^d)}\big)\dd \tau
\\&\lesssim ||u^{l,\omega,n}||_{B^{s}_{2,2}(\R^d)}||u^{l,\omega,n}||_{B^{s+1}_{2,2}(\R^d)}\lesssim n^{-1+\frac12\delta}\cd n^{-1+\frac12\delta}\lesssim n^{-2+\delta}.
\end{split}\end{align}
Combining these estimates \eqref{eq5-1}-\eqref{eq5-4}, we get
\bal\label{eq4-2}
||E^{\omega,n}||_{B^{s-1}_{2,2}(\R^d)}\lesssim n^{-2+\delta}.
\end{align}
Plugging \eqref{eq4}, \eqref{eq4-1} and \eqref{eq4-2} into \eqref{eq-sum}, we have for all $t\in[0,T]$,
\bbal
||v(t)||_{B^{s-1}_{2,2}(\R^d)}&\leq C\int^t_0||v||_{B^{s-1}_{2,2}(\R^d)}\dd \tau+Cn^{-2+\delta}.
\end{align*}
This alongs with Growall's inequallity yields
\bal\label{eq4-3}
||v||_{L^\infty_T(B^{s-1}_{2,2}(\R^d))}\leq Cn^{-2+\delta}.
\end{align}
Noticing that
\bbal
||v||_{L^\infty_T(B^{s+1}_{2,2}(\R^d))}\lesssim ||u_{\omega,n}||_{L^\infty_T(B^{s+1}_{2,2}(\R^d))}+||u^{\omega,n}||_{L^\infty_T(B^{s+1}_{2,2}(\R^d))}\lesssim n,
\end{align*}
and using the interpolation inequality and \eqref{eq4-3}, we have
\bal\label{eq4-10}\begin{split}
||v||_{{L^\infty_T(B^{s}_{2,2}(\R^d))}}\leq ||v||^{\frac12}_{L^\infty_T(B^{s-1}_{2,2}(\R^d))}||v||^{\frac12}_{L^\infty_T(B^{s+1}_{2,2}(\R^d))}\leq Cn^{-\frac12+\frac12\delta}.
\end{split}\end{align}
Combining \eqref{eq4-0} and \eqref{eq4-10}, then there exists some positive constant $c_0$ such that
\begin{align*}\begin{split}
&\quad||u_{1,n}(t)-u_{-1,n}(t)||_{B^{s}_{2,2}(\R^d)}
\\&\geq ||u^{1,n}(t)-u^{-1,n}(t)||_{B^{s}_{2,2}(\R^d)}-C\varepsilon_n
\\&\geq ||u^{h,1,n}(t)-u^{h,-1,n}(t)||_{B^{s}_{2,2}(\R^d)}-C\varepsilon'_n
\\&\geq 2|\sin t|\cdot||n^{-\frac12\delta-s}\phi(\frac{x_1}{n^\delta})\cos(n x_1)\widetilde{\phi}(x_2,\cdots,x_d)||_{B^{s}_{2,2}(\R^d)}-C\varepsilon'_n
\\&\geq c_0|\sin t|\cdot||n^{-\frac12\delta-s}\phi(\frac{x_1}{n^\delta})\cos(n x_1)||_{B^{s}_{2,2}(\R)}||\widetilde{\phi}(x_2,\cdots,x_d)||_{B^{s}_{2,2}(\R^{d-1})}-C\varepsilon'_n,
\end{split}\end{align*}
where
\bbal
\varepsilon_n=n^{-\frac12+\frac12\delta}, \qquad \varepsilon'_n=n^{-\frac12+\frac12\delta}+n^{\frac12\delta-1}.
\end{align*}
Letting $n$ go to $\infty$, we can show that
\begin{align}\label{eq4-11}\begin{split}
\liminf_{n\rightarrow \infty}||u_{1,n}(t)-u_{-1,n}(t)||_{B^{s}_{2,2}(\R^d)}\gtrsim |\sin t|.
\end{split}\end{align}
Notice that $u^{h,1,n}(0,x)=u^{h,-1,n}(0,x)$, we get from Lemma \ref{le3} that
\bal\begin{split}\label{eq4-12}
&\quad||u_{1,n}(0,x)-u_{-1,n}(0,x)||_{B^{s}_{2,2}(\R^d)}
\\&=||u^{l,1,n}(0,x)-u^{l,-1,n}(0,x)||_{B^{s}_{2,2}(\R^d)}\leq Cn^{-1+\frac{1}{2}\delta}\rightarrow 0, \qquad n\rightarrow \infty.
\end{split}\end{align}
Then \eqref{eq4-11} together with \eqref{eq4-12} complete the proof of Theorem \ref{th2}.

\vspace*{1em}
\noindent\textbf{Acknowledgements.}  This work was partially supported by NSFC (No.11801090).

\end{document}